\documentclass{amsart}

\usepackage{amssymb}

\newcommand{\Iso}{\text{Isom}}
\newcommand{\dist}{\text{dist}}

\newtheorem{theorem}{Theorem}[section]
\newtheorem{lemma}[theorem]{Lemma}
\newtheorem{proposition}[theorem]{Proposition}
\newtheorem{corollary}[theorem]{Corollary}

\theoremstyle{definition}
\newtheorem{definition}[theorem]{Definition}
\newtheorem{example}[theorem]{Example}

\theoremstyle{remark}
\newtheorem{remark}[theorem]{Remark}

\numberwithin{equation}{section}

\begin{document}

\title[On geodesic extendibility and the space of compact balls]{On geodesic extendibility and the space of compact balls of length spaces}

\author{Waldemar Barrera}
\address{Facultad de Matem\'aticas, Universidad Aut\'onoma de Yucat\'an, Perif\'erico Norte Tablaje 13615,  C.P. 97110, M\'erida, M\'exico.}
\curraddr{}
\email{bvargas@correo.uady.mx}
\thanks{ORCID: 0000-0001-6885-5556.}

\author{Luis M. Montes de Oca}
\address{Facultad de Matem\'aticas, Universidad Aut\'onoma de Yucat\'an, Perif\'erico Norte Tablaje 13615,  C.P. 97110, M\'erida, M\'exico.}
\curraddr{}
\email{mauricio.montesdeoca@alumnos.uady.mx}
\thanks{}

\author{Didier A. Solis}
\address{Facultad de Matem\'aticas, Universidad Aut\'onoma de Yucat\'an, Perif\'erico Norte Tablaje 13615,  C.P. 97110, M\'erida, M\'exico.}
\curraddr{}
\email{didier.solis@correo.uady.mx}
\thanks{ORCID: 0000-0001-5099-3088}


\begin{abstract}
In this work we study the issue of geodesic extendibility on complete and locally compact metric length spaces. We focus on the geometric structure of the space $(\Sigma (X),d_H)$ of compact balls endowed with the Hausdorff distance and give an explicit isometry  between $(\Sigma (X),d_H)$  and the closed half-space $ X\times \mathbb{R}_{\ge 0}$ endowed with a taxicab metric. Among the applications we establish a group isometry between $\Iso (X,d)$ and $\Iso (\Sigma (X),d_H)$ when $(X,d)$ is a Hadamard space.
\end{abstract}
\maketitle

\section{Introduction}

 In the context of metric length spaces, the study of distance realizing geodesics is very subtle since the standard techniques of global Riemannian geometry do not carry over to the non-smooth scenario in a straightforward way. Les us recall, as an illustrative example, that in CAT(0) spaces geodesics may branch out and yet keep being distance realizing past a branching point.  However, complete and simply connected length spaces with non-positive upper curvature bounds  ---also termed Hadamard spaces---  share many  features with their Riemannian counterparts. For instance,  in such spaces any pair of points can be joined by a unique distance realizing geodesic segment and are homeomorphic to an Euclidean space \cite{Burago,chavel,AB01}, and thus, they may be considered the closest analogs to the classical Euclidean and hyperbolic geometries. 
There has been a significant amount of research on geodesically complete metric length spaces  with upper curvature bounds, specially from the perspective of geometric measure theory \cite{Lytchak,Bertrand} and geometric group theory \cite{caprace,caprace2}.

In this work we deal with a class of length spaces that generalizes Hadamard spaces, namely, locally compact and complete length spaces $(X,d)$ whose maximal geodesics are distance realizing throughout their whole domain of definition. We emphasize that no assumptions are made regarding curvature bounds on  $(X,d)$. 

We focus our attention on the geodesic structure of the space $(\Sigma (X),d_H)$  of compact balls of $X$ endowed with the Hausdorff distance $d_H$. Part of the  motivation behind this endeavor comes from the theory of continua, where this sort of spaces have been studied extensively from the topological point of view in the context of hyperspace theory \cite{Jonard,sergey}, but their geometric structure remains unexplored for the most part.

The present work is organized as follows: in section \ref{sec:prelim} we establish the notation and principal concepts to be used in this work.  Here we introduce the relevant geometric condition relating geodesics and metric balls ---namely, strong geodesic extendibility. Roughly speaking, this property  states that any distance realizing curve going through the center of a metric ball keeps being distance realizing until it exits the ball; or equivalently, that a geodesic admits a geodesic extension of arbitrary length. In section \ref{sec:main}  use this notion to give a description of the geodesic structure of the space $(\Sigma (X),d_H)$ in terms of a taxicab metric (Theorem \ref{main}).  Finally, in section \ref{sec:apps} we use our main result  to prove several different results related to geodesic extendibility. For instance, we show the stability of this property under uniform convergence (Corollary \ref{corollary:uniform}), study its  behavior under quotients of isometric actions (Theorem \ref{theorem:quotient}) and provide a explicit isometry between $\Iso (\Sigma (X),d_H)$ and $\Iso (X,d)$ when the latter is a Hadamard space (Corollary \ref{coro:isom}). 

\section{Preliminaries}\label{sec:prelim}

Let us start by defining some basic concepts and establishing the notation to be used throughout this work. We will be using the notation and definitions found in the standard references \cite{Bridson,Burago,Plaut}.

\subsection{Length spaces and Hausdorff distance}

We will be dealing meanly with geodesic length spaces. Let us recall that a length space is a metric space $(X,d)$ that satisfies $d=d_L$ where $d_L$ is the metric associated to the induced length structure
given by $L(\gamma )=\sup_P\{S(P)\}$, where $P$ denotes a partition $a=t_0<t_1<\cdots <t_{k-1}<t_k=b$ of the closed interval $[a,b]$ and
$S(P)=\sum_{i=1}^kd(\gamma (t_{i-1}),\gamma (t_i))$ denotes de length of the  corresponding polygonal curve approximating $\gamma$. Thus,  a metric space $(X,d)$ is a \textit{length space} if for all $p,q\in X$
\begin{equation*}
d(p,q)=\inf \{ L(\gamma) \mid \gamma \text{ is a curve joining }p\text{ and }q \} .
\end{equation*}
If a curve $\gamma$ joining $p$ and $q$ is distance realizing, --that is, if $L(\gamma )=d(p,q)$-- we call it a \textit{geodesic segment}. If in a length space $(X,d)$ every pair of points can be joined by a geodesic segment we call such a space \textit{geodesic} or \textit{strictly intrinsic}. We will further consider all these segments parameterized with respect to arc length. Let us recall that any complete and locally compact length space is geodesic with compact (metric) closed balls (see Prop. 2.5.22 and Thrm. 2.5.23 in \cite{Burago}).

The \textit{Hausdorff distance} $d_H$ associated to the metric space $(X,d)$ is defined by
    \begin{equation*}
    d_H(A,B) = \inf\{ r: A\subset U_r(B), B\subset U_r(A) \}
     \end{equation*}
where $U_r(A)=\{x\in X: \dist(x,A)<r\} = \displaystyle\cup_{a\in A} B_r(a)$ denotes the tubular neighborhood of $A$ of radius $r$ and $\dist(x,A)=\inf\{ d(x,a) : a\in A\}$. Equivalently, 
\begin{enumerate}
\item $d_H(A,B) = \inf\{r: A\subset \overline{U}_{r}(B), B\subset \overline{U}_{r}(A) \}$, where $\overline{U}_{r}(A) =\{x\in X: \dist(x,A)\leq r\}$.
\item $d_H(A,B)=\displaystyle\max\left\{ \displaystyle\sup_{a\in A} \dist(a,B), \displaystyle\sup_{b\in B} \dist(b,A) \right\}.$
\end{enumerate}
This latter formulation is more convenient for certain computations. For instance, we can readily see that the Hausdorff distance between two compact intervals $I=[a,b]$ and $J=[c,d]$ in $\mathbb{R}$ is given by
\begin{equation*}
d_H([a,b],[c,d]) = \max\{|c-a|,|d-b|\}.
\end{equation*}

\begin{proposition}\label{tubular}
Let $t,s>0$ and $A$ be a compact subset of a complete and locally compact length space $(X,d)$. Then $\overline{U}_{t}(A)$ is compact. Moreover,  $\overline{U}_t(\overline{U}_s(A))=\overline{U}_{t+s}(A)$.
\end{proposition}

Thus, in every complete and locally compact length space we can compute the Hausdorff distance between two closed balls by means of the following formula:
\begin{equation*}
d_{H}(\overline{B}_{t}(x), \overline{B}_{s}(y)) = \inf\{r: \overline{B}_{t}(x) \subset \overline{B}_{s+r}(y), \overline{B}_{s}(y) \subset \overline{B}_{t+r}(x)\}.
\end{equation*}

Finally, given a complete and locally compact metric space $(X,d)$, the space of  the compact balls of $X$ will be denoted by $(\Sigma (X),d_H)$.

\subsection{Strong geodesic extendibility}

Here we establish the notion of strong geodesic extendibility  and provide some examples relating to it. As stated in the previous section we are aiming at a property of geodesics that guarantee that geodesic segments can always be extended as distance realizing curves. Hence, the techniques developed here could be applied to a wide range of length spaces that include Hadamard spaces. 

A \textit{geodesic} in a length space $(X,d)$ is a curve that locally minimizes length. In more precie terms, a geodesic $\gamma : I\subset \mathbb{R}\to X$ satisfies that around any $p=\gamma (t)$ there exists an interval $J$ such that for any $[a,b]\subset J$ the curve $\gamma |_{[a,b]}$ is a geodesic segment (here $I$ and $J$ stand for open intervals). If we assume $(X,d)$ is complete and that any of its geodesics can be extended, then we have as an easy consequence of the metric version of Hopf-Rinow theorem that any maximal geodesic (when parameterized by arc-length) is defined on the whole $\mathbb{R}$. Hence, such spaces are called \textit{geodesically complete}. However, notice that in general geodesic segments may fail to realize distance when extended. Therefore the issue of extending geodesic segments as distance realizing curves is very subtle. 

\begin{example}\label{ex:nobdry}
Let us consider the closed unit disk $\overline{D}\subset \mathbb{R}^2$ with the standard Euclidean distance. Any radius is a geodesic segment that cannot be extended as a geodesic past the endpoint lying on the boundary. More generally, the same is true for geodesic segments joining the interior and the boundary in a Riemannian manifold with non-empty boundary.
\end{example}

\begin{example}\label{Diamond}
 Let us consider the set
\begin{equation*}
X=\{(u,0)\in \mathbb{R}^{2}: |u|\geq 1\} \cup \{(u,v)\in \mathbb{R}^{2}: |u|+|v|=1\},
\end{equation*}
and the intrinsic metric $d$ induced by the standard Euclidian metric in $\mathbb{R}^{2}$. Notice $(X,d)$ is geodesically complete. However, none of the two geodesic segments joining $x=(-1/2,1/2)$ and $y=(1/2,-1/2)$ can be extended as a distance realizing curve.
\end{example}

Recall that in the context of Riemannian geometry, the cut locus of a point $p$ is the set of points in which geodesics emanating from $p$ cease to be distance realizing. A classical result establishes that the cut locus is made up of first conjugate points or else points in which two different geodesics intersect \cite{chavel}. Thus, if there are no points conjugated to $p$ and the exponential map $exp_p$ is injective (as it is the case for Hadamard spaces) then any geodesic emanating from $p$ can be extended as a distance realizing curve. For length spaces, the issue of extendibility is more intricate as there isn't any clear characterization of the points along a geodesic in which it no longer realizes distance.  In \cite{sormani}, several different (though not equivalent) notions of conjugate points and cut loci are studied, leading to interesting results in specific situations.  Following this reference, we say $q$ is in the \textit{minimal cut locus of} $p$,  denoted $MinCut(p)$, if there is a  geodesic segment from $p$ to $q$ that extends to a geodesic which is not minimizing past $q$.  We will be concerned with spaces that can be described essentially as having empty minimal cut loci. This approach has proven fruitful to study convex surfaces \cite{Zem}.

\begin{definition}
Let $(X,d)$ be a complete and locally compact length space. If $MinCut(p)=\emptyset$ and every geodesic segment that starts at $p$ can be extended as a geodesic we say that \textit{strong geodesic extendibility} holds at $p$. If this latter condition holds in any $p\in X$ we say that $(X,d)$  is  \textit{strongly geodesically complete}.  
\end{definition}

\begin{remark}
In virtue of Cartan-Hadamard theorem \cite{Burago,AB01}, any (metric) Hadamard space is strongly geodesically complete. However, there are strongly geodesically complete spaces in which geodesic segments joining two points are not unique. The space $(\mathbb{R}^2,d_T)$, where $d_T(a,b)=|b_1-a_1|+|b_2-a_2|$ is the standard taxicab metric, provides such an example. Notice as well that  $(\mathbb{R}^2,d_T)$ is not a space of bounded curvature. 
\end{remark}

The following equivalent notions of strong geodesic extendibility will be used indistinctively throughout this work.

\begin{proposition}\label{Equivalences}
Let $(X,d)$ be a complete and locally compact length space. Then the following are equivalent:
\begin{enumerate}
\item Strong geodesic extendibility holds at $x$: for every geodesic segment $\gamma:[0,a]\to X$ starting at $x$ and $r>0$ there exists a geodesic segment $\tilde{\gamma}:[0,a+r]\to X$ such that $\tilde{\gamma}|_{[0,a]}=\gamma$.
\item For every geodesic segment $\gamma:[0,a]\to X$ starting at $p$ there exists a geodesic ray\footnote{Recall that a \textit{geodesic ray} is curve $\gamma :[0,\infty )\to X$ whose restriction to any closed interval is a geodesic segment, and thus realizes distance between any pair of its points.} $\tilde{\gamma}:[0,\infty)\to X$ such that $\tilde{\gamma}|_{[0,a]}=\gamma$.
\item For any $r>0$ and any $y\neq x$ there exists a point $p\in\partial\overline{B}_{r}(x)$ and a geodesic segment $\gamma:[0,d(x,y)+r]\to X$ such that $\gamma(0)=y$, $\gamma(d(x,y))=x$ and $\gamma(d(x,y)+r)=p$.
\end{enumerate}
\end{proposition}

\begin{proof}
Assertions (2) $\Rightarrow$ (3) and (3) $\Rightarrow$ (1) are easy to verify. Thus, let us assume that strong geodesic extendibility holds at $x\in X$ and fix a geodesic segment $\gamma:[0,a]\to X$ such that $\gamma(0)=x$. We want to prove that there exists a geodesic ray $\tilde{\gamma}:[a,\infty)\to X$ such that $\tilde{\gamma}|_{[0,a]}=\gamma$. Since strong geodesic extendibility holds at $x$, for the geodesic $\gamma$ and $a>0$ there exists a geodesic segment $\gamma_1:[0,2a]\to X$ such that $\gamma_1|_{[0,a]}=\gamma$. Now, for the geodesic segment $\gamma_1$ and $a>0$ there exists a a geodesic segment $\gamma_{2}:[0,3a]\to X$ such that $\gamma_2|_{[0,2a]}=\gamma_1$ and therefore $\gamma_{2}|_{[0,a]}=\gamma$. Continuing with this construction we get a geodesic segment $\gamma_{n}:[0,(n+1)a]\to X$ such that $\gamma_{n}|_{[0,ka]}=\gamma_{k-1}$ for all $2\leq k\leq n$ and $\gamma_{n}|_{[0,a]}=\gamma$. Define $\tilde{\gamma}:[0,\infty)\to X$ as follows: for $t\in[0,\infty)$ take a large enough $n$ such that $t\leq (n+1)a$ and set $\tilde{\gamma}(t)=\gamma_{n}(t)$. The path $\tilde{\gamma}$ is well defined and it is a geodesic ray by construction. In addition $\tilde{\gamma}|_{[0,a]}=\gamma$ as we required.
\end{proof}

\begin{example}\label{Diamonds}
For every $n\in\mathbb{Z}$ define
\begin{equation*}
C_n = \{(a,b)\in \mathbb{R}^{2}: |x-n|+|y|=1\}.
\end{equation*}
Now, let us consider the space $C = \displaystyle\cup_{n\in\mathbb{Z}} C_{2n}$, endowed with the intrinsic metric induced by the restriction of the Euclidian metric of $\mathbb{R}^{2}$ on $C$. It is not difficult to check that any point of the set
\begin{equation*}
S=\{(2n+1,0)\in \mathbb{R}^{2}: n\in\mathbb{Z}\},
\end{equation*}
satisfies geodesic extendibility. Furthermore, by applying the same arguments as in  Example \ref{Diamond} we can show that no other point in $C$ satisfies geodesic extendibility. Thus  the set of points in which this property holds is discrete and infinite.
\end{example}

A closer look at Examples \ref{Diamond} and \ref{Diamonds} reveals that the set of points in which strong geodesic extendibility holds is  closed. This is a general feature, as the following result shows.

\begin{proposition}
Let $(X,d)$ be a complete and locally compact length space. Then the set of points in which strong geodesic extendibility holds is closed in $(X,d)$.
\end{proposition}

\begin{proof}
Suppose $\{x_n\}_{n=1}^{\infty}$ is a sequence converging to $x$ in which strong geodesic extendibility holds.  Fix $y\neq x$ and $r>0$. Moreover, take $R>\max\{d(x,y),r\}$ large enough such that there exists $N\in\mathbb{N}$ with $x_n\subset \overline{B}_{R}(x)$ for all $n\geq N$. We shall prove that there exists $p\in \partial \overline{B}_{r}(x)$ and a geodesic segment $\gamma:[0, d(x,y)+r]\to X$ such that $\gamma(0)=y$, $\gamma(d(x,y))=x$ and $\gamma(d(x,y)+r)=p$. Because strong geodesic extendibility  holds in $x_n$ for every $n\in\mathbb{N}$, there exists $p_n\in\partial\overline{B}_{r}(x_n)$ and a geodesic segment $\gamma_{n}:[0,d(x_n,y)+r]\to X$ such that $\gamma_n(0)=y$, $\gamma_n(d(x_n,y))=x_n$ and $\gamma_n(d(x_n,y)+r)=p_n$. Moreover, since every geodesic segment $\gamma_n$ has length equals to $d(x_n,y)<R$, then Arzela-Ascoli Theorem (Theorem 2.5.14 of \cite{Burago}) in $\overline{B}_{R}(x)$ ensures the existence of a converging subsequence $\{\gamma_{n_i}\}_{i=1}^{\infty}$. Set $\gamma=\displaystyle\lim_{i\to\infty} \gamma_{n_i}$ and $p=\displaystyle\lim_{i\to \infty} p_{n_i}$. Using Proposition 2.5.17 of \cite{Burago} we get that $\gamma$ is a geodesic segment joining $y$ with $p$. Also, is clear  that $p\in\overline{B}_{r}(x)$ and $\displaystyle\lim_{i\to\infty} \gamma_{n_i}(d(x_i,y)) =x$ since $d(x_i,y)$ converges to $d(x,y)$ when $n\to\infty$. By a suitable reparameterization we can define $\gamma$ on  the interval $[0,d(x,y)+r]$, thus obtaining the desired geodesic segment.
\end{proof}

\section{An isometry between $(\Sigma(X),d_H)$ and $(X\times \mathbb{R}_{\geq 0}, d_T)$}\label{sec:main}

In this section we analyze the geodesic structure of the space of closed balls of a locally compact and complete length space $(X,d)$. First recall that for such an space closed metric balls are compact and the metric $d$ is intrinsic. Thus we can define the space $\Sigma (X)$ of compact balls of $(X,d)$ as
\begin{equation*}
\Sigma(X) = \{\overline{B}_{r}(x): p\in X, r\geq 0\},
\end{equation*}

The key idea in our approach consists in comparing $(\Sigma (X),d_H)$ with the space $X\times\mathbb{R}_{\geq 0}$ endowed with the taxicab metric $d_T$ given by
\begin{equation*}
d_T((x,t), (y,s)) = d(x,y)+ |t-s|.
\end{equation*}
In particular, we are interested in finding necessary and sufficient conditions for the map $f:X\times\mathbb{R}_{\geq 0}\to \Sigma(X)$ given by
\begin{equation*}
f(x,t)=\overline{B}_{t}(x).
\end{equation*}
to be an isometry. Note that $f$ need not be even injective in general, as the following example shows:

\begin{example}
Let us consider the metric space
\begin{equation*}
X=([-1,\infty)\times\{0\})\cup (\{0\}\times [0,1])\subset\mathbb{R}^2
\end{equation*}
with the metric length induced from the Euclidean metric of $\mathbb{R}^2$. Let $P=(-1,0)$ and $Q=(0,1)$, and notice then that $f(P,2)=\overline{B}_2(P)=\overline{B}_2(Q)=f(Q,2)$.
\end{example}

As it turns out,  the lack of injectivity is one of the main obstruction for the map $f$ to be an isometry (see the proof of Theorem \ref{main} below). We begin by showing first that $f$ is a $1$-Lipschitz map.

\begin{proposition}
Let $(X,d)$ a complete and locally compact length space. Then the function $f:(X\times\mathbb{R}_{\geq 0},d_T)\to (\Sigma(X),d_H)$, $f(x,t)=\overline{B}_{t}(x)$ is $1-$Lipschitz.
\end{proposition}

\begin{proof} By definition we have
\begin{equation*}
d_H(f(x,t),f(y,s))  = \max\left\{ \sup_{a\in \overline{B}_{t}(x)} \dist(a,\overline{B}_{s}(y)),
\sup_{b\in \overline{B}_{s}(y)} \dist(b,\overline{B}_{t}(x)) \right\}.
\end{equation*}
Thus, for every $a\in\overline{B}_{t}(x)\setminus\overline{B}_{s}(y)$ we have 
\begin{equation*}
\dist(a,\overline{B}_{s}(y)) = d(a,y)-s\leq  d(a,x) + d(x,y)-s\leq t+d(x,y)-s.
\end{equation*}
 Thus
\begin{equation*}
\sup_{a\in \overline{B}_{t}(x)} \dist(a,\overline{B}_{s}(y)) \leq d(x,y) + t-s.
\end{equation*}
Proceeding in a similar way, we have
\begin{equation*}
\sup_{b\in \overline{B}_{s}(y)} \dist(b,\overline{B}_{t}(x))\leq d(x,y)+s-t.
\end{equation*}
Hence
\begin{equation*}
d_H(f(x,t),f(y,s)) \leq \max\{d(x,y)+t-s,d(x,y)+s-t\} = d_T((x,t),(y,s)).
\end{equation*}
\end{proof}

In order to get a better insight, let us assume for the time being that $f$ is an isometry and consider the closed ball $\overline{B}_r(x)$. Let $y\neq x$ then
\begin{equation*}
\dist(y,\overline{B}_{r}(x)) \leq d(x,y) \leq \displaystyle\sup_{a\in\overline{B}_{r}(x)} d(a,y),
\end{equation*}
thus
\begin{equation*}
d_H(f(x,r),f(y,0)) = d_H(\overline{B}_{r}(x),\{y\}) =  \displaystyle\sup_{a\in\overline{B}_{r}(x)} d(a,y).
\end{equation*}
Since $f$ is an isometry we have $\displaystyle\sup_{a\in\overline{B}_{r}(x)} d(a,y) = d(x,y) + r$. Furthermore, since  $\overline{B}_{r}(x)$ is a compact set, there exists $p\in \overline{B}_r(x)$ with
\begin{equation*}
d(p,y) = \displaystyle\sup_{a\in\overline{B}_{r}(x)} d(a,y) = d(x,y) + r.
\end{equation*}
Now we show that $p\in \partial \overline{B}_r(x)$. Indeed, notice that
\begin{equation*}
r + d(x,y) = d(p,y) \leq d(p,x) + d(x,y) \leq r + d(x,y),
\end{equation*}
hence $d(p,x)=r$.

Finally, let us observe that $x,y,p$ are all collinear. Let  $\alpha:[0,d(x,y)]:\to X$ and $\beta:[0,r]\to X$
be geodesic segments ---parameterized with respect to arc length--- such that $\alpha(0)=y$, $\alpha(d(x,y))=x=\beta(0)$ and $\beta(r)=p$. Recall such paths exist since $(X,d)$ is geodesic. Consider then the path $\gamma:[0,d(x,y)+r]\to X$ given by $\gamma=\alpha *\beta$ and let us show that for all $t,s\in[0,\ell+r]$ we have $d(\gamma(t),\gamma(s))=|t-s|$. If $t,s\in[0,d(x,y)]$ or $t,s\in[d(x,y), d(x,y)+r]$ then there is nothing to prove, since $\alpha$ and $\beta$ are geodesic segments parameterized by arc length. Thus let $t\in [0,d(x,y)]$ and $s\in[d(x,y),d(x,y)+r]$ and notice
\begin{align*}
d(y,p) &= d(\gamma(0),\gamma(t)) + d(\gamma(t),\gamma(s)) + d(\gamma(s),\gamma(\ell+r)) \\
&\leq d(\gamma(0),\gamma(t)) + d(\gamma(t),\gamma(\ell)) + d(\gamma(\ell),\gamma(s)) + d(\gamma(s),\gamma(\ell+r)) \\
&= \ell+r = d(y,x) + d(x,p) = d(y,p).
\end{align*}
It then follows that 
\begin{equation*}
d(\gamma(t),\gamma(s))=s-t=|t-s|,
\end{equation*} 
which in turn implies that  $\gamma$ is a geodesic segment. 

According to Proposition \ref{Equivalences}, the above discussion shows that if $f$ is an isometry then $(X,d)$ is strongly geodesically complete.   Notice however, that this is no longer true if we assume that $f$ is merely injective, as the next example shows.

\begin{example}\label{ex:noshoot}
Consider the half space $X=\mathbb{R}\times \mathbb{R}_{\geq 0}$ and the metric $d$ given by the restriction to $X$ of the Euclidian metric from $\mathbb{R}^{2}$. As pointed out in Example \ref{ex:nobdry}, this metric space is not strongly geodesically complete, although $f$ is clearly injective in this case.
\end{example}

We are now ready to establish the main result characterizing the space of compact balls in strongly geodesically complete spaces.

\begin{theorem}\label{main}
Let $(X,d)$ be a locally compact and complete length space. Then the map $f:(X\times\mathbb{R}_{\geq 0},d_T)\to (\Sigma(X),d_H)$, $f(x,t)=\overline{B}_{t}(x)$ is an isometry if and only if $(X,d)$ is strongly geodesically complete.
\end{theorem}

\begin{proof}
Let $x,y\in X$, $t,s\geq 0$, $\ell=d(x,y)$. We first prove that $f$ is injective. Let us assume that $d_H(\overline{B}_t(x), \overline{B}_{s}(y) ) = 0$. Thus
    \begin{equation*}
    \displaystyle\sup_{a\in\overline{B}_{t}(x)} \dist(a,\overline{B}_{s}(y)) =0,
    \end{equation*}
    that is, $\dist(a,\overline{B}_{s}(y))=0$ for all $a\in \overline{B}_{t}(x)$. Hence we have  $\overline{B}_{t}(x)\subset \overline{B}_{s}(y)$. We can show in an analogous way that $\overline{B}_{s}(y) \subset \overline{B}_{t}(x)$ and thus $\overline{B}_{s}(y) = \overline{B}_{t}(x)$. Since $(X,d)$ is strongly geodesically complete, then there exist geodesic segments $\alpha:[0,\ell+t]\to X$, $\beta:[0,\ell+s]\to X$, and points $p\in\partial\overline{B}_{t}(x)$, $q\in\partial \overline{B}_{s}(y)$ such that $\alpha(0)=y=\beta(\ell)$, $\alpha(\ell)=x=\beta(0)$, $p=\alpha(\ell+t)$ y $\beta(\ell+s)=q$. Moreover, since $p\in\partial\overline{B}_{t}(x)=\partial\overline{B}_{s}(y)$, then we have $d(y,p)=s$. Hence
    \begin{equation*}
    t + \ell = d(p,x) + d(x,y) = d(p,y) = s.
    \end{equation*}
  By a similar argument we have $s+\ell = t$. It then follows that  $s=t$ and $\ell=0$, hence $(x,t)=(y,s)$. As a consequence we have that $f$ is injective.

 Let us consider now $d_H(\overline{B}_t(x), \overline{B}_{s}(y) ) >0 $. Without loss of generality we can further assume
    \begin{equation*}
    \displaystyle\sup_{a\in\overline{B}_{t}(x)} \dist(a,\overline{B}_{s}(y)) >0.
    \end{equation*}
    Therefore we have
    \begin{align*}
    \sup\limits_{a\in\overline{B}_{t}(x)} \dist(a,\overline{B}_{s}(y)) &=
    \sup \{ d(a,y) | {a\in\overline{B}_{t}(x)\setminus \overline{B}_{s}(y) }\} -s \\
    &\leq
     \sup \{ d(a,x) + d(x,y) | {a\in\overline{B}_{t}(x)\setminus \overline{B}_{s}(y) } \} - s \\
    &\leq d(x,y) + t -s.
    \end{align*}
   On the other hand, since $(X,d)$ is strongly geodesically complete, there exist a geodesic segment $\gamma:[0,\ell+t]\to X$ and a point $p\in \partial\overline{B}_{t}(x)$ such that $\gamma(0)=y$, $\gamma(\ell)=x$ and $\gamma(\ell+t)=p$. Thus $d(p,y) = d(p,x) + d(x,y) = d(x,y) +t$ and hence
    \begin{equation*}
    \displaystyle\sup_{a\in\overline{B}_{t}(x)} \dist(a,\overline{B}_{s}(y)) = d(x,y)  + t - s.
    \end{equation*}
    Moreover, if
    \begin{equation*}
    \displaystyle\sup_{b\in\overline{B}_{s}(y)} \dist(b,\overline{B}_{t}(x)) >0,
    \end{equation*}
    then by a similar argument we end up with
    \begin{equation*}
    \displaystyle\sup_{b\in\overline{B}_{s}(y)} \dist(b,\overline{B}_{t}(x)) = d(x,y)  +s-t
    \end{equation*}
    and therefore
    \begin{align*}
    d_H(\overline{B}_{t}(x), \overline{B}_{s}(y)) &=
    \max\left\{
    \displaystyle\sup_{a\in\overline{B}_{t}(x)} \dist(a,\overline{B}_{s}(y)),
    \displaystyle\sup_{b\in\overline{B}_{s}(y)} \dist(b,\overline{B}_{t}(x))
    \right\} \\
    &=  \max\{ d(x,y) + t-s , d(x,y) + s-t\} = d(x,y) + |t-s|.
    \end{align*}
    In the case
    \begin{equation*}
    \displaystyle\sup_{b\in\overline{B}_{s}(y)} \dist(b,\overline{B}_{t}(x)) =0
    \end{equation*}
    we have $\overline{B}_{s}(y) \subset \overline{B}_{t}(x)$, so $s\leq t$. Thus \begin{equation*}
d_H(\overline{B}_{t}(x), \overline{B}_{s}(y)) =\max\left\{ \displaystyle\sup_{a\in\overline{B}_{t}(x)} \dist(a,\overline{B}_{s}(y)), 0 \right\} = d(x,y) + |t-s|.
\end{equation*}
Hence $f$ is an isometry.
\end{proof}

\section{Applications}\label{sec:apps}

\subsection{Stability of strong geodesic completeness}

Let us recall that a sequence of metric spaces $(X_n, d_n)$ converges uniformly to a metric space $(X,d)$ if there exist metrics $\bar{d}_{n}$ such that every metric space $(X_n,d_n)$ is isometric to $(X,\bar{d}_n)$ and the sequence of metrics $\bar{d}_n$ converges uniformly to $d$ \cite{Burago}, that is,
\begin{equation*}
\lim_{n\to\infty} \displaystyle\sup_{(a,b)\in X\times X} |\bar{d}_n(a,b)-d(a,b)| = 0
\end{equation*}
Since every metric space $(X_n,d_n)$ is isometric to $(X,\bar{d}_n)$ we might even consider that $X_n=X$ and $\bar{d}_n= d_n$, just to simplify the notation.

We want to establish a relation between the Hausdorff distance in $X$ and the Hausdorff distance in $X_n$. In order to do this, let us denote by $d_H^{n}$ the Hausdorff distance in $(X_n,d_n)$ and the closed ball with center $x$ an radius $t$ by $\overline{B}_{t}^{n}(x)$.

\begin{lemma}\label{ballconvergence}
Let $(X_n,d_n)$ be a sequence of length spaces that converges uniformly to a length space $(X,d)$ and fix $x\in X$, $t\in \mathbb{R}_{\geq 0}$. Then for every $\varepsilon>0$ there is an $N\in\mathbb{N}$ such that
\begin{equation*}
\overline{B}_{t}(x) \subseteq \overline{B}_{t+\varepsilon}^{n}(x) \ \mbox{and} \ \overline{B}_{t}^{n}(x) \subseteq \overline{B}_{t+\varepsilon}(x),\ \forall n\ge N.
\end{equation*}
\end{lemma}

\begin{proof}
Take $\varepsilon>0$. From  uniform convergence there exists $N\in\mathbb{N}$ such that
\begin{equation*}
-\varepsilon \leq d_n(a,b) - d(a,b) \leq \varepsilon,
\end{equation*}
for all $a,b\in X$ and for all $n\ge N$. Thus for every $p\in \overline{B}_{t}(x)$ we have $d(p,x)\leq t$ and therefore
\begin{equation*}
d_n(p,x) \leq d(p,x) + \varepsilon \leq t + \varepsilon,
\end{equation*}
then $d_n(p,x)\leq t+\varepsilon$, it means $p\in\overline{B}_{t+\varepsilon}^{n}(x)$. A similar procedure shows $\overline{B}_{t}^{n}(x) \subseteq \overline{B}_{t+\varepsilon}(x)$ and the proof is complete.
\end{proof}

\begin{proposition}\label{ballconvergencelimit}
Let $(X_n,d_n)$ be a sequence of locally compact length spaces that converges uniformly to a locally compact length space $(X,d)$. Then for every $x,y\in X$ and $t,s\in \mathbb{R}_{\geq 0}$ we have
\begin{equation*}
d_H(\overline{B}_{t}(x),\overline{B}_{s}(y)) =  \displaystyle\lim_{n\to\infty} d_{H}^{n}(\overline{B}_{t}^{n}(x) , \overline{B}_{t}^{n}(y)),
\end{equation*}
\end{proposition}

\begin{proof}
Since $(X,d)$ is a length space we have
\begin{equation*}
d_H(\overline{B}_{t}(x),\overline{B}_{s}(y)) = \inf\{r : \overline{B}_{t}(x)\subseteq \overline{B}_{s+r}(y), \overline{B}_{s}(y)\subseteq \overline{B}_{t+r}(x) \}.
\end{equation*}
Take $\varepsilon>0$ and $\delta=\varepsilon /2$. We will proof that there exists $N\in\mathbb{N}$ such that
\begin{equation*}
|d_H(\overline{B}_{t}(x),\overline{B}_{s}(y)) - d_{H}^{n}(\overline{B}_{t}^n(x),\overline{B}_{s}^n(y))| \leq \varepsilon,
\end{equation*}
for every $n\geq N$. Using Lemma \ref{ballconvergence} there exists $N\in \mathbb{N}$ such that for all $n\ge N$
\begin{equation}\label{eq:1}
\overline{B}_{t}(x) \subseteq \overline{B}_{t+\delta}^{n}(x) \ \mbox{and} \
\overline{B}_{t}^{n}(x) \subseteq \overline{B}_{t+\delta}(x),
\end{equation}
\begin{equation}\label{eq:2}
\overline{B}_{s}(y) \subseteq \overline{B}_{s+\delta}^{n}(y) \ \mbox{and} \
\overline{B}_{s}^{n}(y) \subseteq \overline{B}_{s+\delta}(y).
\end{equation}
Fix $n$ such that $n\geq N$ and suppose $r$ satisfies $\overline{B}_{t}^{n}(x) \subseteq \overline{B}_{s+r}^{n}(y)$ and $\overline{B}_{s}^{n}(y) \subseteq \overline{B}_{t+r}^{n}(x)$. Using eq. \ref{eq:1} we get  $\overline{B}_{t}(x) \subseteq \overline{B}_{t+\delta}^{n}(x) \subseteq \overline{B}_{s+r+\delta}^{n}(y)$, therefore $\overline{B}_{t}(x)\subseteq \overline{B}_{s+r+\delta}^{n}(y)$. Now, using eq. \ref{eq:2} we have
 $ \overline{B}_{s+r+\delta}^{n}(y) \subseteq \overline{B}_{s+r+\delta+\delta}(y) = \overline{B}_{s+r+\varepsilon}(y)$,  thus $\overline{B}_{t}(x) \subseteq \overline{B}_{s+r+\varepsilon}(y)$. A similar computation leads to $    \overline{B}_{s}(y) \subseteq \overline{B}_{t+r+\varepsilon}(x)$.
    These last contentions imply
    \begin{equation*}
    d_H(\overline{B}_{t}(x),\overline{B}_{s}(y)) \leq r + \varepsilon,
    \end{equation*}
    so taking the infimum over all $r$ satisfying our assumptions we get
    \begin{equation*}
    d_H(\overline{B}_{t}(x),\overline{B}_{s}(y)) - d_{H}^{n}(\overline{B}_{t}^n(x),\overline{B}_{s}^n(y)) \leq \varepsilon .
    \end{equation*}
Analogously, 
    \begin{equation*}
    d_{H}^{n}(\overline{B}_{t}^n(x),\overline{B}_{s}^n(y)) - d_H(\overline{B}_{t}(x),\overline{B}_{s}(y)) \leq \varepsilon,
    \end{equation*}
    which is equivalent to $-\varepsilon \leq d_H(\overline{B}_{t}(x),\overline{B}_{s}(y)) - d_{H}^{n}(\overline{B}_{t}^n(x),\overline{B}_{s}^n(y))$. Thus
\begin{equation*}
|d_H(\overline{B}_{t}(x),\overline{B}_{s}(y)) - d_{H}^{n}(\overline{B}_{t}^n(x),\overline{B}_{s}^n(y))| \leq \varepsilon,
\end{equation*}
as we desired.
\end{proof}

As a corollary from Proposition \ref{ballconvergencelimit} we can now establish the stability of  strong geodesic completeness under uniform limits.

\begin{corollary}\label{corollary:uniform}
Let $(X_n,d_n)$ a sequence of locally compact and strongly geodesically complete length spaces that converges uniformly to a locally compact length space $(X,d)$. Then $(X,d)$ is strongly geodesically complete.
\end{corollary}

\begin{proof}
Just observe that
\begin{align*}
d_H(\overline{B}_{t}(x),\overline{B}_{s}(y)) &= \displaystyle\lim_{n\to\infty} d_H^{n}(\overline{B}_{t}^{n}(x),\overline{B}_{s}^{n}(y)) 
= \displaystyle\lim_{n\to\infty} \left( d_n(x,y) + |t-s| \right) \\
&= d(x,y) + |t-s| .
\end{align*}
Thus by  Theorem \ref{main} $(X,d)$ is geodesically complete.
\end{proof}

\subsection{The maximum metric $d_\infty$.}

Besides the standard product metric $d_{X\times Y}$, there are other choices of metrics in the cartesian product $X\times Y$ that are useful in certain geometric applications. For instance, the following metric has been studied in the context of hyperbolic complex manifolds \cite{Kobayashi}.

For any two metric spaces $(X,d_X)$ and $(Y,d_Y)$ it is possible to define a metric on $X\times Y$ by
\begin{equation*}
d_{\infty}((x,y),(a,b)) = \max\{d_X(x,a), d_Y(y,b)\}.
\end{equation*}
We call $d_\infty$ the \textit{maximum metric} in $X\times Y$. If $X$ and $Y$ are complete and locally compact metric spaces then $(X\times Y,d_{\infty})$ is also complete and locally compact. Moreover, if $X$ and $Y$ are length spaces, then $(X\times Y,d_{\infty})$ is a length space. In fact, for any two points $(x,y),(a,b)\in X\times Y$, the point $(z,c)$ ---where $z$ and $c$ are midpoints of $x,y$ and $a,b$ respectively--- is a midpoint for $(x,y)$ and $(a,b)$. Since the space is complete we conclude that $X\times Y$ is a length space.

As an application of our previous results we show that the cartesian product of two strongly geodesically complete length spaces also has this property relative to the maximum metric. In order to be able to apply Theorem \ref{main}, we need to find first  the Hausdorff distance between two closed balls in $(X\times Y,d_{\infty})$.

\begin{proposition}\label{HausdorffBallsInftyMetric}
For every $(x,y),(a,b)\in X\times Y$ and $t,s\in\mathbb{R}_{\geq 0}$ we have
\begin{equation*}
d_{H}^{\infty}\left( \overline{B}_{t}(x,y), \overline{B}_{s}(a,b)  \right) = \max\left\{ d_H^{X}(\overline{B}_{t}(x),\overline{B}_{s}(a)), d_H^{Y}(\overline{B}_{t}(y),\overline{B}_{s}(b))  \right\},
\end{equation*}
where $d_{H}^{\infty}$, $d_{H}^{X}$ and $d_{H}^{Y}$ are the Hausdorff distances in $(X\times Y,d_{\infty})$, $X$ and $Y$, respectively.
\end{proposition}

\begin{proof}
As in the proof of Proposition \ref{ballconvergencelimit} we will use that
\begin{equation*}
d_{H}^{\infty}\left( \overline{B}_{t}(x,y), \overline{B}_{s}(a,b)  \right) = \inf\left\{ r: \overline{B}_{t}(x,y)\subset \overline{B}_{s+r}(a,b),
\overline{B}_{s}(a,b)\subset \overline{B}_{t+r}(x,y)   \right\} .
\end{equation*}
Suppose $r$ satisfies $\overline{B}_{t}(x,y)\subseteq \overline{B}_{s+r}(a,b)$ and
$\overline{B}_{s}(a,b)\subseteq \overline{B}_{t+r}(x,y)$. Since 
\begin{equation*}
\overline{B}_{t}(x)\times \overline{B}_{t}(y) =  \overline{B}_{t}(x,y)\subseteq \overline{B}_{s+r}(a,b) = \overline{B}_{s+r}(a) \times \overline{B}_{s+r}(b),
\end{equation*}
then we have 
\begin{equation*}
\overline{B}_{t}(x) \subseteq \overline{B}_{s+r}(a) \ \mbox{and} \
\overline{B}_{t}(y)\subseteq \overline{B}_{s+r}(b).
\end{equation*}
A similar procedure shows that
\begin{equation*}
\overline{B}_{s}(a) \subseteq \overline{B}_{t+r}(x) \ \mbox{and} \ \overline{B}_{s}(b) \subset \overline{B}_{t+r}(y).
\end{equation*}
Thus the last two relations yield $d_{H}^{X}(\overline{B}_{t}(x),\overline{B}_{s}(a)) \leq r$ and  $d_{H}^{Y}(\overline{B}_{t}(y),\overline{B}_{s}(b)) \leq r$, which implies
\begin{equation*}
\max\left\{ d_{H}^{X}(\overline{B}_{t}(x),\overline{B}_{s}(a)),  d_{H}^{Y}(\overline{B}_{t}(y),\overline{B}_{s}(b)) \right\} \leq r.
\end{equation*}
Because of the way we choose $r$ we conclude
\begin{equation*}
\max\left\{ d_{H}^{X}(\overline{B}_{t}(x),\overline{B}_{s}(a)),  d_{H}^{Y}(\overline{B}_{t}(y),\overline{B}_{s}(b)) \right\} \leq d_{H}^{\infty}\left( \overline{B}_{t}(x,y), \overline{B}_{s}(x,y) \right).
\end{equation*}
Now suppose $r_1$ and $r_2$ satisfy $\overline{B}_{t}(x) \subseteq \overline{B}_{s+r_1}(a)$, $\overline{B}_{s}(a) \subseteq \overline{B}_{t+r_1}(x)$, $\overline{B}_{t}(y) \subseteq \overline{B}_{s+r_2}(b)$ and $\overline{B}_{s}(b) \subseteq \overline{B}_{t+r_2}(y)$. Take $r=\max\{r_1,r_2\}$, then
    \begin{align*}
    \overline{B}_{t}(x,y) &= \overline{B}_{t}(x)\times \overline{B}_{t}(y) 
    \subseteq  \overline{B}_{s+r_1}(a) \times \overline{B}_{s+r_2}(b) \\
    &\subseteq  \overline{B}_{s+r}(a) \times \overline{B}_{s+r}(b) 
    = \overline{B}_{s+r}(a,b).
    \end{align*}
Therefore $\overline{B}_{t}(x,y) \subseteq \overline{B}_{s+r}(a,b)$. A similar computation proves that $\overline{B}_{s}(a,b) \subseteq \overline{B}_{t+r}(x,y)$. Thus
    \begin{equation*}
    d_{H}^{\infty}\left( \overline{B}_{t}(x,y), \overline{B}_{s}(a,b) \right) \leq r.
    \end{equation*}
    Because of the way we take $r$ we conclude
    \begin{equation*}
    d_{H}^{\infty}\left( \overline{B}_{t}(x,y), \overline{B}_{s}(a,b) \right) \leq \max\left\{ d_H^{X}(\overline{B}_{t}(x),\overline{B}_{s}(a)), d_H^{Y}(\overline{B}_{t}(y),\overline{B}_{s}(b))  \right\} .
    \end{equation*}
\end{proof}

\begin{corollary}
Let $(X,d_X)$ and $(Y,d_Y)$ be locally compact and strongly geodesically complete length spaces, then $(X\times Y, d_{\infty})$ is strongly geodesically complete. \end{corollary}

\begin{proof}
Just note that
\begin{align*}
d_{H}^{\infty}\left( \overline{B}_{t}(x,y), \overline{B}_{s}(a,b)  \right) &= \max\left\{ d_H^{X}(\overline{B}_{t}(x),\overline{B}_{s}(a)), d_H^{Y}(\overline{B}_{t}(y),\overline{B}_{s}(b))  \right\} \\
&= \max\{ d_X(x,a) + |t-s|, d_Y(y,b) + |t-s| \} \\
&= \max\{d_X(x,a),d_Y(y,b)\} + |t-s| \\
&= d_{\infty}((x,y),(a,b)) + |t-s|.
\end{align*}
\end{proof}

\subsection{Isometric actions}

Isometric actions play a key role in various areas of geometry. It is often the case that certain properties of such actions, understood as symmetries of a geometric object, give enough information in order to describe ---or even characterize--- the object they are acting on.  For instance, in the context of length spaces with bounded curvature we have the following result: if an Alexandrov space have an isometry group of maximal size, then it is isometric to a Riemannian manifold \cite{GalazGuijarro}. Another example appears in \cite{Berestovskii}, where Berestovskii proves that a finite dimensional homogeneous metric space with curvature bounded by below is a smooth manifold.

In this section we show how isometric actions on a strongly geodesically complete length space can be lifted to isometric actions in its space of compact balls and further describe such quotients.

Let us denote by $G\curvearrowright X$ an isometric action of $G<\Iso(X,d)$ on $X$ and  by $(\mathcal{H}(X), d_H)$ the space of all compact subsets of $X$ endowed with the Hausdorff distance. For every $g\in G$ we define $\mathcal{H}[g]: \mathcal{H}(X)\to \mathcal{H}(X)$ as $\mathcal{H}[g](K)=g(K)$. It is not hard to prove that $\mathcal{H}[g]$ is an isometry of $(\mathcal{H}(X),d_H)$ and $\mathcal{H}[f\circ g^{-1}] = \mathcal{H}[f]\circ \mathcal{H}[g]^{-1}$ for every $f,g\in G$. This implies that
\begin{equation*}
\mathcal{H}[G] = \{ \mathcal{H}[g]: g\in G \}
\end{equation*}
is a subgroup of $\Iso(\mathcal{H}(X),d_H)$. In particular $\mathcal{H}[G]$ acts by isometries on $\Sigma(X)$.

If the isometric action  $G\curvearrowright X$ is proper ---that is, for each $x\in X$ there exists $r>0$ such that the set $\{g\in G: g(B_r(x))\cap B_{r}(x)\}$ is finite--- 
then ${X}/{G}$ is a metric space \cite{Bridson}. As the next result shows, strong geodesic completeness enables us to transfer properness from $(X,d)$ to $(\Sigma (X),d_H)$.

\begin{proposition}\label{propershooting}
Let $(X,d)$ be a locally compact and strongly geodesically complete length space and $G\curvearrowright X$ a proper isometric action. Then the isometric action from $\mathcal{H}[G]$ on $\Sigma(X)$ is proper.
\end{proposition}

\begin{proof}
Take $\overline{B}_r(x)\in \Sigma(X)$ and let us denote the closed ball with center $K$ and radius $R$ in $(\Sigma(X),d_H)$ by $\overline{B}^{H}_{R}(K)$. Because the action $G\curvearrowright X$ is proper, there exists $R>0$ such that
\begin{equation*}
H=\{g\in G: g(B_{R}(x))\cap B_{R}(x)\neq \varnothing\}
\end{equation*}
is a finite set, and suppose that $\mathcal{H}[g]\in \mathcal{H}[G]$ satisfies
\begin{equation*}
\mathcal{H}[g](B^{H}_{R}(\overline{B}_{r}(x))) \cap B^{H}_{R}(\overline{B}_{r}(x)) \neq \varnothing.
\end{equation*}
On the other hand
\begin{align*}
\mathcal{H}[g](B^{H}_{R}(\overline{B}_{r}(x))) \cap B^{H}_{R}(\overline{B}_{r}(x)) &=  g(B^{H}_{R}(\overline{B}_{r}(x))) \cap B^{H}_{R}(\overline{B}_{r}(x)) \\
&= B^{H}_{R}(g(\overline{B}_{r}(x))) \cap B^{H}_{R}\overline{B}_{r}(x)) \\
&= B^{H}_{R}(\overline{B}_{r}(g(x))) \cap B^{H}_{R}(\overline{B}_{r}(x)),
\end{align*}
then there exists $\overline{B}_{s}(y) \in B^{H}_{R}(\overline{B}_{r}(g(x))) \cap B^{H}_{R}(\overline{B}_{r}(x))$ and therefore this ball satisfies
\begin{equation*}
d_H(\overline{B}_{s}(y), \overline{B}_{r}(g(x))) < R \mbox{\ and \ } d_H(\overline{B}_{s}(y), \overline{B}_{r}(x)) < R.
\end{equation*}
Further, as $(X,d)$ is strongly geodesically complete we have the equalities
\begin{align*}
d_H(\overline{B}_{s}(y), \overline{B}_{r}(g(x))) &= d(y,g(x)) + |s-r|\\
d_H(\overline{B}_{s}(y), \overline{B}_{r}(x)) &= d(y,x) + |s-r|,
\end{align*}
then $d(y,g(x))<R$ and $d(y,x)<R$. As a consequence of these two last inequalities we have $y\in g(B_{R}(x))\cap B_{R}(x)$, which implies $g\in H$. Hence
\begin{equation*}
\{ \mathcal{H}[g]\in\mathcal{H}[G] :  \mathcal{H}[g](B^{H}_{R}(\overline{B}_{r}(x))) \cap B^{H}_{R}(\overline{B}_{r}(x)) \neq \varnothing \} \subset \{ \mathcal{H}[g]: g\in H\},
\end{equation*}
and thus the action $\mathcal{H}[G]\curvearrowright \Sigma(X)$ is proper.
\end{proof}

Using Proposition \ref{propershooting} we conclude ${\Sigma(X)}/{\mathcal{H}[G]}$ is a metric space where the metric is given by
\begin{equation*}
d_{\mathcal{H}[G]}(\mathcal{H}[G](A), \mathcal{H}[G](B)) = \displaystyle\inf_{A,B\in \Sigma(X)} \{d_H(A,B)\},
\end{equation*}
for $A,B\in \Sigma(X)$. Moreover, if $X$ is a length space, then ${X}/{G}$ and ${\Sigma(X)}/{\mathcal{H}[G]}$ are both length spaces.

\begin{lemma}
Let $(X,d)$ be a locally compact and strongly geodesically complete length space and $G\curvearrowright X$. Then, for every $\overline{B}_{s}(y) \in \mathcal{H}[G](\overline{B}_{t}(x))$ we have $y\in G(x)$ and $s=t$.
\end{lemma}

\begin{proof}
Because $\overline{B}_{s}(y)$ belongs to the orbit of $\overline{B}_{t}(x)$ there exists $g\in G$ such that
\begin{equation*}
\overline{B}_{s}(y) = \mathcal{H}[g](\overline{B}_{t}(x)) = g(\overline{B}_{t}(x)) = \overline{B}_{t}(g(x)).
\end{equation*}
Since $(X,d)$ is strongly geodesically complete we get $s=t$ and $y=g(x)$ as we desired.
\end{proof}

\begin{theorem}\label{theorem:quotient}
Let $(X,d)$ be a locally compact and strongly geodesically complete length space satisfying and $G\curvearrowright X$ a proper action. Then ${\Sigma(X)}/{\mathcal{H}[G]}$ is isometric to $({X}/{G})\times \mathbb{R}_{\geq 0}$ endowed with the taxicab metric $d_T$.
\end{theorem}

\begin{proof}
Fix $x_0,y_0\in X$ and $t_0,s_0\in \mathbb{R}_{\geq 0}$, thus
\begin{align*}
d(\mathcal{H}(\overline{B}_{t_0}(x_0)) , \mathcal{H}(\overline{B}_{s_0}(y_0))) &= \displaystyle\inf\{d_H(\overline{B}_t(x),\overline{B}_s(y))\}  \\
&= \displaystyle\inf\{d(g_x(x_0),g_y(y_0)) + |t_0-s_0|\} \\
&= d_G(G(x),G(y)) + |t_0-s_0|.
\end{align*}
\end{proof}

\begin{remark}
Not every metric quotient is strongly geodesically complete even when the base space is. For instance consider $\mathbb{R}$ with its standard metric and the group generated by the isometry $x\mapsto x+2\pi$. The quotient space is isometric to the circle $\mathbb{S}^{1}$ which does not satisfies strong geodesic completeness since the map $f$ from Theorem \ref{main} is not injective.
\end{remark}

\subsection{The group $\Iso (\Sigma (X),d_H)$}

One of the main goals of a geometric theory consists in the description of the structure of its group of isometries. In the case of Riemannian geometry, Myers and Steenrod proved in \cite{MS} that the group of isometries of a Riemannian manifold is a Lie group, and the same statement remains true for length spaces which are locally compact, with finite Hausdorff dimension and curvature bounded by below \cite{Fukaya-Yamaguchi}. 

Our last results pertain the description of the group of isometries $\Iso (\Sigma (X),d_H)$ when $(X,d)$ is a metric Hadamard space.

We note that strong geodesic completeness is closely related to the existence of a correspondence between isometries of $(X,d)$ and isometries of $(\Sigma (X),d_H)$. Indeed, if $(X,d)$ is strongly geodesically complete then any isometry of $f : (X,d)\to (X,d)$ gives rise to a natural isometry  $F :(\Sigma (X),d_H)\to (\Sigma (X),d_H)$ as follows: for any $B\in \Sigma (X)$, define $F(B)=f(B)$. Since $f$ is an isometry, it sends closed balls to closed balls and $F(\overline{B}_r(x))=\overline{B}_r(f(x))$. Then
 \begin{align*}
 d_H(F(\overline{B}_r(x)),F(\overline{B}_s(y))) &= d_H(\overline{B}_r(f(x)),\overline{B}_s(f(y))) = |r-s|+d(f(x),f(y))\\ &= |r-s| +d(x,y) = d_H(\overline{B}_r(x),\overline{B}_s(y)) .
 \end{align*}
Hence $F$ is an isometry of $(\Sigma (X),d_H)$.

On the other hand, additional properties can be used to show that ---in some specific cases--- in fact all isometries of $(\Sigma (X),d_H))$ arise in this way. For instance, in \cite{gruber1} it is shown that any isometry of $(\Sigma (\mathbb{R}^n),d_H)$ that sends one point sets to one point sets is a rigid motion. (See also \cite{gruber2}). In our setting, the relevant geometric property is  uniqueness of midpoints.

\begin{remark} \label{rem:uniquemp}
Let us notice that even if $(X,d)$ is strongly geodesically complete, geodesics in $(\Sigma (X),d_H)$ might not have unique midpoints. In fact, $\overline{B}_{t}(z)$ is a midpoint  for the pair $(\overline{B}_r(p),\overline{B}_s(q))$, where $t=(r+s)/2$ and $z$ is any midpoint of $d(p,q)$. However, if $(X,d)$ in addition has unique midpoints, then  $(\overline{B}_r(p),\overline{B}_s(q))$ has a unique midpoint in $(\Sigma (X), d_H)$ only when $r=s$.
\end{remark}

\begin{theorem}
Let $(X,d)$ be a locally compact and strongly geodesically complete length space. Then, every isometry of $(X,d)$ gives rise to an isometry of $(\Sigma (X),d_H)$. Further, if $(X,d)$ has unique midpoints then every isometry of $(\Sigma (X),d_H)$ gives rise to an isometry of $(X,d)$. 
\end{theorem}

\begin{proof}  Let $F\in\Iso (\Sigma (X),d_H)$ and for any $t\ge 0$ let us denote 
\begin{equation*}
\Sigma_t=
\{\overline{B}_t(x)\ \mid \ x\in X \}\cong \{ t\}\times X.
\end{equation*}
 We first proceed to show that $F$ maps balls of the same radius to balls of the same radius. In other words, given $t\ge 0$ there exists $r\ge 0$ such that $F(\Sigma_t )\subset \Sigma_r$. Thus consider $F(\overline{B}_t(p))=\overline{B}_r(\hat{p} )$, $F(\overline{B}_t(q))=\overline{B}_s(\hat{q} )$. Hence
\begin{equation*}
d(p,q)=d_H(\overline{B}_t(p),\overline{B}_t(q))=d_H(F(\overline{B}_t(p)),F(\overline{B}_t(q)))=|r-s|+d(\hat{p},\hat{q}).
\end{equation*}
Since $(\Sigma (X), d_H)$ is intrinsic, there exist midpoints between $F(\overline{B}_t(p))$ and $F(\overline{B}_t(q))$. Let $B\in \Sigma (X)$  such a midpoint. Thus
\begin{equation*}
d_H(F(\overline{B}_t(p)),B)=d_H(p,q)/2=d_H(F(\overline{B}_t(q)),B).
\end{equation*}
Further, since $F^{-1}\in\Iso \, (\Sigma (X),d_H)$ we have
\begin{equation*}
d_H(\overline{B}_t(p),F^{-1}(B))=d_H(p,q)/2=d_H(\overline{B}_t(q),F^{-1}(B))
\end{equation*}
Let $F^{-1}(B)=\overline{B}_u(a)$, thus
\begin{align*}
d_H(p,q) &=d(p,q)  \le d(p,a)+d(a,q)\\ &= [d_H(\overline{B}_t(p),F^{-1}(B))-|t-u|]+[d_H(\overline{B}_t(q), F^{-1}(B))-|t-u|]\\ &= d(p,q)-2|t-u| = d_H(p,q)-2|t-u|.
\end{align*}
It follows that $t=u$, so $B=F(\overline{B}_t(a))$. Thus 
\begin{align*}
d(p,a) &= d_H(\overline{B}_t(p),\overline{B}_t(a))=d_H(F(\overline{B}_t(p)),B)=d(p,q)/2,\\
d(q,a) &= d_H(\overline{B}_t(q),\overline{B}_t(a))=d_H(F(\overline{B}_t(q)),B)=d(p,q)/2,
\end{align*}
Since  $(X,d)$ has unique midpoints we then have that $a$ is the unique midpoint for $(p,q)$. As a consequence, the midpoint $B=F(\overline{B}_t(a))$ between $F(\overline{B}_t(p))$ and $F(\overline{B}_t(q))$ is also unique. By Remark \ref{rem:uniquemp} above we then have $r=s$ and hence $F(\Sigma_t)\subset \Sigma_r$. Further, given $B'\in \Sigma_r$ then $F^{-1}(B')\in \Sigma_t$. It follows that $F(\Sigma_t)=\Sigma_r$. Let $F(\{x\})=F(\overline{B}_0(x))=\overline{B}_R(\hat{x})$ and define $f:X\to X$ by $f(x)=\hat{x}$. We have just shown that $f$ is a distance preserving (hence injective) surjection. 
\end{proof}

Since in a Hadamard space any pair of points can be joined by a unique geodesic segment, in virtue of the above theorem we can establish in this case the equivalence between the isometries of $(\Sigma (X),d_H)$ and $\Iso (X,d)$. 

\begin{corollary}\label{coro:isom}
Let $(X,d)$ be a Hadamard space. Then $\Iso (\Sigma (X),d_H)\cong \Iso (X,d)$. In particular,
$\Iso(\Sigma (\mathbb{R}^n),d_H)\cong O(n)$ and $\Iso(\Sigma (\mathbb{H}^n),d_H)\cong O_0(n,1)$.
\end{corollary}

\section*{Acknowledgements}
W. Barrera was partially supported by grants Conacyt-SNI 45382 and P/PFCE-2017-31MSU0098J-13. D. A. Solis was partially supported by grants Conacyt-SNI 38368, P/PFCE-2017-31MSU0098J-13 and UADY-CEA-SAB011-2017. D. A. Solis want to acknowledge the kind hospitality of CIMAT- M\'erida, where part of this work was developed during a sabbatical leave.

\bibliographystyle{amsplain}

\end{document}